\documentclass[12pt]{amsart}
\usepackage{amssymb,amsthm,amsmath,amstext,amsxtra}
\usepackage{fullpage}
\usepackage{lmodern}   % for font sizes
\usepackage{pbsi}
\usepackage[T1]{fontenc}
\usepackage{bm}        % allows bold italic letters in math mode
\usepackage{mathtools} % allows more extendible arrows
\usepackage[all]{xy}
\usepackage{array, booktabs}
\usepackage{hyperref}
\hypersetup{colorlinks=true,urlcolor=blue,citecolor=blue,linkcolor=blue}
\usepackage{enumerate}
\usepackage{enumitem}
\usepackage{colonequals}
\usepackage{comment, todonotes}
\usepackage{multirow}
\usepackage{color}

\numberwithin{equation}{section}

\theoremstyle{plain}
\newtheorem*{theorema*}{Theorem A}
\newtheorem*{theoremb*}{Theorem B}
\newtheorem{theorem}[equation]{Theorem}
\newtheorem*{theorem*}{Theorem}

\newtheorem{lemma}[equation]{Lemma}
\newtheorem{corollary}[equation]{Corollary}
\newtheorem{conjecture}[equation]{Conjecture}

\theoremstyle{definition}

\newtheorem{example}[equation]{Example}

\theoremstyle{remark}
\newtheorem{remark}[equation]{Remark}

\DeclarePairedDelimiter\abs{\lvert\hspace{0.1ex}}{\rvert}

\newcommand{\Z}{\mathbb Z}
\newcommand{\Q}{\mathbb Q}
\newcommand{\R}{\mathbb R}
\newcommand{\F}{\mathbb F}

\newcommand{\A}{\mathbb A}
\newcommand{\C}{\mathbb C}

\newcommand{\Gal}{\mathrm{Gal}}

\newcommand{\Frob}{\mathrm{Frob}}

\newcommand{\defi}[1]{\textsf{#1}} 	% for defined terms

\newcommand{\legen}[2]{\left(\frac{#1}{#2}\right)}

\DeclareMathOperator{\As}{Asai}
\DeclareMathOperator{\Asai}{\As}
\DeclareMathOperator{\GO}{GO}

\DeclareMathOperator{\M}{M}

\DeclareMathOperator{\GL}{GL}
\DeclareMathOperator{\Nm}{Nm}

\DeclareMathOperator{\ord}{ord}

\newcommand{\psmod}[1]{~(\textup{\text{mod}}~{#1})}

\newcommand{\id}{\mathrm{id}}

\DeclareMathOperator{\Res}{\mathrm{Res}}

\newcommand{\frakp}{\mathfrak{p}}

\newcommand{\frakM}{\mathfrak{M}}
\newcommand{\frakN}{\mathfrak{N}}

\newcommand{\Qbar}{\Q^{\textup{al}}}
\newcommand{\Qlbar}{\Q^{\textup{al}}_\ell}

\let\eps\varepsilon

\newcommand{\Sym}{\operatorname{Sym}}

\renewcommand{\d}{{\mathrm d}}

\renewcommand{\Im}{\operatorname{Im}}
\newcommand{\balpha}{\boldsymbol\alpha}
\newcommand{\bbeta}{\boldsymbol\beta}
\newcommand{\calq}{(q-1)}

\usepackage{hyperref}
\hypersetup{pdftitle={Special hypergeometric motives and their $L$-functions: Asai recognition},pdfauthor={Lassina Demb\'el\'e, Alexei Panchishkin, John Voight, Wadim Zudilin}}
\hypersetup{colorlinks=true,linkcolor=blue,anchorcolor=blue,citecolor=blue}

\begin{document}

\title{Special hypergeometric motives and their $L$-functions: Asai recognition}

\author{Lassina Demb\'el\'e}
%\address{Department of Mathematics, King's College London, Strand WC2R~2LS, London, United Kingdom}
\address{Department of Mathematics, Dartmouth College, 6188 Kemeny Hall, Hanover, NH 03755, USA}
\email{lassina.dembele@gmail.com}

\author{Alexei Panchishkin}
\address{Institut Fourier, Universit\'e Grenoble-Alpes, B.P.~74, 38402 St.-Martin d'H\`eres, France}
\email{alexei.pantchichkine@univ-grenoble-alpes.fr}
\urladdr{\url{https://www-fourier.ujf-grenoble.fr/~panchish/}}

\author{John Voight}
\address{Department of Mathematics, Dartmouth College, 6188 Kemeny Hall, Hanover, NH 03755, USA}
\email{jvoight@gmail.com}
\urladdr{\url{http://www.math.dartmouth.edu/~jvoight/}}

\author{Wadim Zudilin}
\address{Department of Mathematics, IMAPP, Radboud University, PO Box 9010, 6500\,GL Nijme\-gen, Netherlands}
\email{w.zudilin@math.ru.nl}
\urladdr{\url{http://www.math.ru.nl/~wzudilin/}}

\begin{abstract}
We recognize certain special hypergeometric motives, related to and inspired by the discoveries of Ramanujan more than a century ago, as arising from Asai $L$-functions of Hilbert modular forms.
\end{abstract}

\date{\today}

\maketitle

%\tableofcontents

\section{Introduction}
\label{sec1}

\subsection*{Motivation}

The generalized hypergeometric functions are a familiar player in arithmetic and algebraic geometry.
They come quite naturally as periods of certain algebraic varieties, and consequently they encode important information about the invariants of these varieties.  Many authors have studied this rich interplay, including Igusa \cite{Igusa}, Dwork \cite{padic}, and Katz \cite{Ka90}.  More recently, authors have considered \emph{hypergeometric motives} (HGMs) defined over $\Q$, including Cohen \cite{Coh15}, Beukers--Cohen--Mellit \cite{BCM15}, and Roberts--Rodriguez-Villegas--Watkins \cite{RVW}.  A hypergeometric motive over $\Q$ arises from a parametric family of varieties with certain periods (conjecturally) satisfying  a hypergeometric differential equation; the construction of this family was made explicit by Beukers--Cohen--Mellit \cite{BCM15} based on work of Katz \cite{Ka90}.  Following the analogy between periods and point counts (Manin's ``unity of mathematics'' \cite{clemens}), counting points on the reduction of these varieties over finite fields is accomplished via \emph{finite field hypergeometric functions}, a notion originating in work of Greene \cite{Gr87} and Katz~\cite{Ka90}.  These finite sums are analogous to truncated hypergeometric series in which Pochhammer symbols are replaced with Gauss sums, and they provide an efficient mechanism for computing the $L$-functions of hypergeometric motives.  (Verifying the precise connection to the hypergeometric differential equation is usually a difficult task, performed only in some particular cases.)

In this paper, we illustrate some features of hypergeometric motives attached to particular arithmetically significant hypergeometric identities for $1/\pi$ and $1/\pi^2$.  To motivate this study, we consider the hypergeometric function
\begin{equation}  \label{eqn:3F20}
{}_3F_2\biggl(\begin{matrix} \frac12, \, \frac14, \, \frac34 \\ 1, \, 1 \end{matrix} \biggm| z \biggr)
=\sum_{n=0}^\infty
\frac{(\frac{1}{2})_n(\frac{1}{4})_n(\frac{3}{4})_n}
{n!^3}\, z^n,
\end{equation}
where we define the Pochhammer symbol (rising factorial) by
\begin{equation}
(\alpha)_n \colonequals \frac{\Gamma(\alpha+n)}{\Gamma(\alpha)}=\begin{cases}
\alpha(\alpha+1)\dotsb(\alpha+n-1), &\text{for $n\ge1$}; \\
1, &\text{for $n=0$}.
\end{cases}
\end{equation}
%where for $n \geq 1$, we define the Pochhammer symbol $(\alpha)_n=\alpha(\alpha+1)\dotsb(\alpha+n-1)$.
Ramanujan \cite[eq.~(36)]{Ra14} more than a century ago proved the delightful identity
\begin{equation} 
\sum_{n=0}^\infty\frac{(\frac12)_n(\frac14)_n(\frac34)_n}{n!^3}(28n+3)\biggl(-\frac1{48}\biggr)^n
=\frac{16}{\pi\sqrt3}
\label{rama1}
\end{equation}
involving a linear combination of the hypergeometric series \eqref{eqn:3F20} and its $z$-derivative (a different, but contiguous hypergeometric function).  Notice the practicality of this series for computing the quantity on the right-hand side of \eqref{rama1}, hence for computing~$1/\pi$ and $\pi$ itself.

The explanation for the identity \eqref{rama1} was already indicated by Ramanujan: the hypergeometric function can be parametrized by modular functions (see \eqref{eqn:can0tau} below), and the value $-1/48$ arises from evaluation at a complex multiplication (CM) point!  Put into the framework above, we observe that the HGM of rank $3$ with parameters $\balpha=\{\frac12,\frac14,\frac34\}$ and $\bbeta=\{1,1,1\}$ corresponds to the Fermat--Dwork pencil of quartic K3 surfaces of generic Picard rank $19$ defined by the equation
\begin{equation} \label{eqn:XpsiK3}
X_{\psi} \colon x_0^4+x_1^4+x_2^4+x_3^4 = 4\psi x_0x_1x_2x_3
\end{equation}
whose transcendental $L$-function is related to the symmetric square $L$-function attached to a classical modular form (see Elkies--Sch\"utt \cite{ES}).  At the specialization $z = \psi^{-4}=-1/48$, the K3 surface is singular, having Picard rank $20$; it arises as the Kummer surface of $E \times \tau(E)$, where $E$ is the elliptic $\Q$-curve LMFDB label \href{http://www.lmfdb.org/EllipticCurve/2.2.12.1/144.1/b/1}{\textsf{144.1-b1}} defined over $\Q(\sqrt{3})$ attached to the CM order of discriminant $-36$, and $\tau(\sqrt{3})=-\sqrt{3}$.   The corresponding classical modular form $f$ with LMFDB label \href{http://www.lmfdb.org/ModularForm/GL2/Q/holomorphic/144/2/c/a/}{\textsf{144.2.c.a}} has CM, and we have the identity
\begin{equation} 
L(T(X),s) = L(f,s,\Sym^2)
\end{equation}
where $T(X)$ denotes the transcendental lattice of $X$ (as a Galois representation).  The rare event of CM explains the origin of the formula \eqref{eqn:3F20}: for more detail, see Example \ref{exm:K3ex} below.

\subsection*{Main result}

With this motivation, we seek in this paper to explain similar hypergeometric Ramanujan-type formulas for $1/\pi^2$ in higher rank.
Drawing a parallel between these examples, our main result is to experimentally identify that the $L$-function of certain specializations of hypergeometric motives (coming from these formulas) have a rare property: they arise from Asai $L$-fun\-ctions of Hilbert modular forms of weight $(2,4)$ over real quadratic fields.  

For example, consider the higher rank analogue
\begin{equation}
\sum_{n=0}^\infty\frac{(\frac12)_n(\frac13)_n(\frac23)_n(\frac14)_n(\frac34)_n}{n!^5}(252n^2+63n+5)\biggl(-\frac1{48}\biggr)^n
\overset?=\frac{48}{\pi^2}
\label{rama2}
\end{equation}
given by Guillera \cite{Gu03}; the question mark above a relation indicates that it has been experimentally observed, but not proven.  Here, we suggest that \eqref{rama2} is `explained' by the existence of a Hilbert modular form $f$ over $\Q(\sqrt{12})$ of weight $(2,4)$ and level $(81)$ in the sense that we experimentally observe that
\begin{equation} \label{eqn:HGMeqn}
L(H(\tfrac12,\tfrac13,\tfrac23,\tfrac14,\tfrac34;1,1,1,1,1\,|\,{-1/48}), s) \overset?= \zeta(s-2)L(f,s+1,\Asai) 
\end{equation}
where notation is explained in section \ref{sec3}.  (By contrast, specializing the hypergeometric $L$-series at other values $t \in \Q$ generically yields a primitive $L$-function of degree $5$.)  Our main result, stated more generally, can be found in Conjecture \ref{mainconj}.

In spite of a visual similarity between Ramanujan's $_3F_2$ formula \eqref{rama1} for $1/\pi$ and Guillera's $_5F_4$ formula \eqref{rama2} for $1/\pi^2$,
the structure of the underlying hypergeometric motives is somewhat different. Motives attached to $_3F_2$ hypergeometric functions are reasonably well understood (see e.g.\ Zudilin \cite[Observation~4]{Zu18}), and we review them briefly in Section~\ref{sec2}.
By contrast, the $_5F_4$ motives associated with similar formulas had not been linked explicitly to modular forms.   In Conjecture \ref{mainconj}, we propose that they are related to Hilbert modular forms, and we experimentally establish several other formulas analogous to \eqref{eqn:HGMeqn}.

More generally, for a hypergeometric family, we expect interesting behavior (such as a formula involving periods) when the motivic Galois group at a specialization is smaller than the motivic Galois group at the generic point.  We hope that experiments in our setting leading to this kind of explanation will lead to further interesting formulas and, perhaps, a proof.  

% However, even such identification of HGMs and algebraic data is successfully executed, one still wonders about what arithmetic it corresponds to.
% Recent achievements allow us to identify the $L$-functions of rank~2 HGMs with those of classical modular forms, both computationally and rigorously.
% At the same time, higher rank HGMs appeal to numerous conjectural correspondences between algebraic varieties and automorphic forms, like the paramodularity conjecture.
% These are really hard to tackle, even computationally \cite{BPPTVY18}.

\subsection*{Organization}

The paper is organized as follows.  After a bit of setup in section \ref{sec2}, we quickly review hypergeometric motives in section \ref{sec3}.  In section \ref{sec5} we discuss Asai lifts of Hilbert modular forms, then in section \ref{sec6} we exhibit the conjectural hypergeometric relations.  We conclude in section \ref{sec7} with some final remarks.

\subsection*{Acknowledgements}

This project commenced during Zudilin's visit in the Fourier Institute in Grenoble in June 2017,
followed by the joint visit of Demb\'el\'e, Panchishkin, and Zudilin to the Max Planck Institute for Mathematics in Bonn in July 2017, followed by collaboration between Voight and Zudilin during the trimester on \emph{Periods in Number Theory, Algebraic Geometry and Physics} at the Hausdorff Research Institute for Mathematics in Bonn in March--April 2018.
We thank the staff of these institutes for providing such excellent conditions for research.  

Demb\'el\'e and Voight were supported by a Simons Collaboration grant (550029, to Voight).

It is our pleasure to thank Frits Beukers, Henri Cohen, Vasily Golyshev, Jes\'us Guillera, G\"unther Harder, Yuri Manin, Anton Mellit, David Roberts, Alexander Varchenko, Fernando Rodriguez-Villegas, Mark Watkins and Don Zagier
for valuable feedback, stimulating discussions, and crucial observations.  The authors would also like to thank the anonymous referee for their stimulating report, insightful comments, and helpful suggestions.

\section{Hypergeometric functions}
\label{sec2}

In this section, we begin with some basic setup.  For $\alpha_1,\dots,\alpha_d \in \Q$ and $\beta_1,\dots,\beta_{d-1} \in \Q_{>0}$, define the \defi{generalized hypergeometric function} 
\begin{equation}
{}_dF_{d-1}\biggl(\begin{matrix}
\alpha_1, \, \alpha_2, \, \dots, \, \alpha_d \\
\beta_1, \, \dots, \, \beta_{d-1} \end{matrix} \biggm|z\biggr)
\colonequals \sum_{n=0}^\infty
\frac{(\alpha_1)_n(\alpha_2)_n\dotsb(\alpha_d)_n}
{(\beta_1)_n\dotsb(\beta_{d-1})_n}\,
\frac{z^n}{n!}\,.
\label{hfun}
\end{equation}
These functions possess numerous features that make them unique in the class of special functions.
It is convenient to abbreviate \eqref{hfun} as
\begin{equation}
{}_dF_{d-1}(\balpha,\bbeta\mid z)
=F(\balpha,\bbeta\mid z),
\label{hfun2}
\end{equation}
where $\balpha=\{\alpha_1,\dots,\alpha_d\}$ and
$\bbeta=\{\beta_1,\dots,\beta_d\}=\{\beta_1,\dots,\beta_{d-1},1\}$
are called the \defi{parameters} of the hypergeometric function: they are multisets (that is, sets with possibly repeating elements), with the additional element $\beta_d=1$ introduced to reflect the appearance of $n!=(1)_n$ in the denominator in~\eqref{hfun}.
The hypergeometric function \eqref{hfun} satisfies a linear homogeneous differential equation of order~$d$:
\begin{equation}
D(\balpha,\bbeta;z): \quad
\biggl(z\prod_{j=1}^d\biggl(z\frac{\d}{\d z}+\alpha_j\biggr)-\prod_{j=1}^d\biggl(z\frac{\d}{\d z}+\beta_j-1\biggr)\biggr)y=0.
\label{hde}
\end{equation}

Among many arithmetic instances of the hypergeometric functions, there are those that can be parameterized by modular functions.
One particular example, referenced in the introduction, is
\begin{equation} \label{eqn:can0tau}
{}_3F_2\biggl(\begin{matrix} \frac12, \, \frac14, \, \frac34 \\ 1, \, 1 \end{matrix} \biggm|\rho(\tau)\biggr)
=\biggl(\frac{\eta(\tau)^{16}}{\eta(2\tau)^8}+64\,\frac{\eta(2\tau)^{16}}{\eta(\tau)^8}\biggr)^{1/2},
\end{equation}
for $\tau \in \C$ with $\Im(\tau)>0$, where 
\begin{equation} 
\rho(\tau) \colonequals \frac{256\eta(\tau)^{24}\eta(2\tau)^{24}}{(\eta(\tau)^{24}+64\eta(2\tau)^{24})^2} 
\end{equation}
and 
$\eta(\tau)=q^{1/24}\prod_{j=1}^\infty(1-q^j)$ denotes the Dedekind eta function with $q=e^{2\pi i\tau}$.
Taking the CM point $\tau=(1+3\sqrt{-1})/2$, we obtain $\rho(\tau)=-1/48$ and the evaluation \cite[Example 3]{GZ13}
\begin{equation} 
{}_3F_2\biggl(\begin{matrix} \frac12, \, \frac14, \, \frac34 \\ 1, \, 1 \end{matrix} \biggm|-\frac1{48}\biggr)
=\frac{\sqrt{2}}{\pi\,3^{5/4}}\biggl(\frac{\Gamma(\frac14)}{\Gamma(\frac34)}\biggr)^2.
\label{CSF}
\end{equation}
As indicated by Ramanujan \cite{Ra14}, CM evaluations of hypergeometric functions like \eqref{CSF} are accompanied
by formulas for $1/\pi$, like \eqref{rama1} given in the introduction. 

\begin{remark}
Less is known about the conjectured congruence counterpart of \eqref{CSF},
\begin{equation}
\sum_{n=0}^{p-1}\frac{(\frac12)_n(\frac14)_n(\frac34)_n}{n!^3}\biggl(-\frac1{48}\biggr)^n
\overset?\equiv b_p\pmod{p^2}
\label{CSF1}
\end{equation}
for primes $p\ge5$, where
\begin{equation}
b_p \colonequals \begin{cases}
2(x^2-y^2) & \text{if}\; p\equiv1\psmod{12}, \; p=x^2+y^2 \;\text{with}\; 3\mid y; \\
-(x^2-y^2) & \text{if}\; p\equiv5\psmod{12}, \; p=\tfrac12(x^2+y^2) \;\text{with}\; 3\mid y; \\
0 & \text{if}\; p\equiv3\psmod{4}.
\end{cases}
\label{ap43}
\end{equation}
The congruence \eqref{CSF1}
%(see also \cite[Conjecture 4.14]{Su18} for its modulo $p^3$ version)
is in line with a general prediction of Roberts--Rodriguez-Villegas \cite{RV17}, though stated there for $z=\pm1$ only.  
\end{remark}

%  As explained in the introduction, a companion to \eqref{CSF} given by Ramanujan \cite[eq.~(36)]{Ra14} reads
% \begin{equation}
% \sum_{n=0}^\infty\frac{(\frac12)_n(\frac14)_n(\frac34)_n}{n!^3}(28n+3)\biggl(-\frac1{48}\biggr)^n
% =\frac{16}{\pi\sqrt3}.
% \label{rama1}
% \end{equation}
% in which a linear combination of the series \eqref{CSF} and its $z$-derivative (which is a different but contiguous hypergeometric function) is involved.

\begin{comment}
Our principal source of investigation here are Ramanujan's formula \eqref{rama1} and its higher dimensional analogue
\begin{equation}
\sum_{n=0}^\infty\frac{(\frac12)_n(\frac13)_n(\frac23)_n(\frac14)_n(\frac34)_n}{n!^5}(252n^2+63n+5)\biggl(-\frac1{48}\biggr)^n
\overset?=\frac{48}{\pi^2}
\label{rama2}
\end{equation}
given by Guillera \cite{Gu03} (observe the divisibility of $252$ by~$28$).
The two possess \cite{Zu09} supercongruence counterparts
\begin{align*}
\sum_{n=0}^{p-1}\frac{(\frac12)_n(\frac14)_n(\frac34)_n}{n!^3}(28n+3)\biggl(-\frac1{48}\biggr)^n
&\overset?\equiv3\biggl(\frac{-3}p\biggr)p\pmod{p^3}
\\ \intertext{and}
\sum_{n=0}^{p-1}\frac{(\frac12)_n(\frac13)_n(\frac23)_n(\frac14)_n(\frac34)_n}{n!^5}(252n^2+63n+5)\biggl(-\frac1{48}\biggr)^n
&\overset?\equiv5p^2\pmod{p^5}
\end{align*}
for primes $p\ge5$, where $\bigl(\frac{\cdot}p\bigr)$ denotes the Legendre(--Jacobi--Kronecker) symbol.
The first of the two congruences is proven modulo $p^2$ in~\cite{Gu18}.
\end{comment}

Ramanujan's and Ramanujan-type formulas for $1/\pi$ corresponding to rational values of $z$ are tabulated in \cite[Tables 3--6]{CC12}.
Known $_5F_4$ identities for $1/\pi^2$ are due to Guillera \cite{Gu02,Gu03,Gu06,Gu11,Gu19}, also in collaboration with Almkvist \cite{AG12} and Zudilin \cite{GZ12}.
We list the corresponding hypergeometric data $\balpha$ and $z$ for them in Table~\ref{tab1}, we have
$\bbeta=\{1,1,1,1,1\}$ in all these cases.

\begin{equation} \label{tab1}\addtocounter{equation}{1} \notag
\begin{gathered}
{\renewcommand{\arraystretch}{1.2}
\begin{tabular}{r|c|c|c}
\# & $\balpha$ & $z$ & reference \\
\hline
\hline
 1 & $\{\tfrac12,\tfrac12,\tfrac12,\tfrac12,\tfrac12\}$ & $-1/2^2 \vphantom{|^{0^0}}$ & \cite[p.~46]{Gu06}, \cite{AG12} \\
 2 & $\{\tfrac12,\tfrac12,\tfrac12,\tfrac12,\tfrac12\}$ & $-2^2$ & \cite[eq.~(2)]{GZ12} \\
 3 & $\{\tfrac12,\tfrac12,\tfrac12,\tfrac12,\tfrac12\}$ & $-1/2^{10}$ & \cite[p.~603]{Gu02}, \cite{AG12} \\
 4 & $\{\tfrac12,\tfrac12,\tfrac12,\tfrac12,\tfrac12\}$ & $-2^{10}$ & \cite[eq.~(8)]{GZ12} \\
 5 & $\{\tfrac12,\tfrac12,\tfrac12,\tfrac13,\tfrac23\}$ & $(3/4)^3$ & \cite[eq.~(17)]{Gu11}, \cite{AG12} \\
 6 & $\{\tfrac12,\tfrac12,\tfrac12,\tfrac13,\tfrac23\}$ & $-3^3$ & \cite[eq.~(36)]{Gu19} \\
 7 & $\{\tfrac12,\tfrac15,\tfrac25,\tfrac35,\tfrac45\}$ & $-5^5/2^8$ & \cite[eq.~(39)]{Gu19} \\ 
 8 & $\{\tfrac12,\tfrac12,\tfrac12,\tfrac14,\tfrac34\}$ & $1/2^4$ & \cite[p.~603]{Gu02}, \cite{AG12} \\
 9 & $\{\tfrac12,\tfrac13,\tfrac23,\tfrac14,\tfrac34\}$ & $-1/48$ & \cite[eq.~(2-3)]{Gu03}, \cite{AG12} \\
10 & $\{\tfrac12,\tfrac13,\tfrac23,\tfrac14,\tfrac34\}$ & $-3^3/2^4$ & \cite[eq.~(9)]{GZ12} \\
11 & $\{\tfrac12,\tfrac13,\tfrac23,\tfrac16,\tfrac56\}$ & $-(3/4)^6$ & \cite[Table 3]{AG12} \\
12 & $\{\tfrac12,\tfrac13,\tfrac23,\tfrac16,\tfrac56\}$ & $(3/5)^6$ & \cite[eq.~(1-1)]{AG12} \\
13 & $\{\tfrac12,\tfrac13,\tfrac23,\tfrac16,\tfrac56\}$ & $-1/80^3$ & \cite[eq.~(2-4)]{Gu03}, \cite{AG12} \\
14 & $\{\tfrac12,\tfrac14,\tfrac34,\tfrac16,\tfrac56\}$ & $-1/2^{10}$ & \cite[eq.~(2-2)]{Gu03}, \cite{AG12} \\
15 & $\{\tfrac12,\tfrac18,\tfrac38,\tfrac58,\tfrac78\}$ & $1/7^4$ & \cite[eq.~(2-5)]{Gu03}, \cite{AG12} \\
16 & $\{\tfrac12,\tfrac12,\tfrac12,\tfrac13,\tfrac23\}$ & $3^3((-1\pm \sqrt5)/2)^{15}$ & \cite[Table 3]{AG12}, \cite[eq.~(38)]{Gu19} 
\end{tabular}} \\
\text{Table \ref{tab1}: Hypergeometric data for Guillera's formulas for $1/\pi^2$}
\end{gathered}
\end{equation}

\begin{remark}
Some other entries in Table~\ref{tab1} nicely pair up with Ramanujan's and Ramanu\-jan-type formulas for $1/\pi$  \cite{Gu03}. Apart from case \#9 from Table~\ref{tab1} discussed above, we highlight another instance \cite[eq.~(2-4)]{Gu03}:
\begin{equation}
\sum_{n=0}^\infty\frac{(\frac12)_n(\frac13)_n(\frac23)_n(\frac16)_n(\frac56)_n}{n!^5}(5418n^2+693n+29)\biggl(-\frac1{80^3}\biggr)^n
\overset?=\frac{128\sqrt5}{\pi^2}
\label{rama2b}
\end{equation}
underlying entry \#13, which shares similarities with the Ramanujan-type formula
\begin{equation}
\sum_{n=0}^\infty\frac{(\frac12)_n(\frac16)_n(\frac56)_n}{n!^3}(5418n+263)\biggl(-\frac1{80^3}\biggr)^n
=\frac{640\sqrt{15}}{3\pi}.
\label{rama1b}
\end{equation}
\end{remark}

\begin{remark} \label{par-structure0}
The specialization points $z$ in Table~\ref{tab1} exhibit significant structure: writing $z=a/c$ and $1-z=b/c$, so that $a+b=c$, we already see $abc$-triples of good quality!  But more structure is apparent: see Remark \ref{par-structure}.  
\end{remark}

\section{Hypergeometric motives}
\label{sec3}

%Good Euler factors for hypergeometric $L$-functions \cite{Ka90} admit fast $p$-adic implementation due to Henri Cohen.

In this section, we quickly introduce the theory of hypergeometric motives over $\Q$.  
% As mentioned in the introduction, a hypergeometric motive \cite{Vi17} is a family of algebraic varieties with a period  are related to solutions of the corresponding differential equation \eqref{hde}. A candidate for the variety is based on counting points over finite fields\,---\,computing the Frobenius operators; it is given in \cite{Ka90,Gr87} and, in a much greater generality, in \cite{BCM15}.

\subsection*{Definition}

Analogous to the generalized hypergeometric function \eqref{hfun}, a hypergeometric motive is specified by \defi{hypergeometric data}, consisting of two multisets
${\boldsymbol\alpha}=\{\alpha_1, \dots ,\alpha_d\}$ and ${\boldsymbol \beta}=\{\beta_1, \dots ,\beta_d\}$ with $\alpha_j,\beta_j \in \Q \cap (0,1]$ satisfying 
$\balpha \cap \bbeta = \emptyset$ and $\beta_d=1$.  
% The premises are exactly the same as in the case of hypergeometric function \eqref{hfun2}
% and the related hypergeometric differential equation~\eqref{hde}, where $z$ is the additional parameter
% with respect to which we differentiate. In what follows, a (somewhat different) parameter $\lambda$ is used to parameterise
% hypergeometric motives in a family.
Herein, we consider only those hypergeometric motives that are \defi{defined over $\mathbb Q$},
which means that the polynomials 
\begin{equation}
\prod_{j=1}^d ( t - e^{2\pi i\alpha_j}) \quad \text{and} \quad \prod_{j=1}^d ( t - e^{2\pi i\beta_j}) 
\end{equation}
have coefficients in $\mathbb Z$---that is, they are products of cyclotomic polynomials.

Let $q$ be a prime power that is coprime to the least common denominator of $\balpha \cup \bbeta$, and let $\mathbb F_q$ be a finite field with $q$ elements.  Let $\omega \colon \F_q^\times \to \C^\times$ be a generator of the character group on $\F_q^\times$, and let $\psi_q \colon \F_q \to \C^\times$ be a nontrivial (additive) character.  For $m \in \Z$, define the \defi{Gauss sum}
\begin{equation}
g(m) \colonequals \sum_{x\in \mathbb F_q^\times}\omega(x)^m\psi_q(x);
\end{equation}
then $g(m)$ is periodic in $m$ with period $q - 1=\#\mathbb F_q^\times$.  

When $\alpha_j(q-1),\beta_j(q-1) \in \Z$ for all $j$, we define the \defi{finite field hypergeometric sum} for $t \in \F_q^\times$ by
\begin{equation}
\label{FHS}
H_q(\balpha,\bbeta\mid t)=\frac{1}{1-q}\sum_{m=0}^{q-2}\omega((-1)^d t)^m
\prod_{j=1}^d\frac{g(m+\alpha_j\calq)\,g(-m-\beta_j\calq)}{g(\alpha_j\calq)\,g(-\beta_j\calq)}
\end{equation}
by direct analogy with the generalized hypergeometric function.  More generally, Beukers--Cohen--Mellit \cite[Theorem 1.3]{BCM15} have extended this definition to include all prime powers $q$ that are coprime to the least common denominator of $\balpha \cup \bbeta$.  

\begin{comment}
Though this definition works under the hypothesis $\alpha_i\calq$, $\beta_j\calq\in  \mathbb Z$ for all $i$ and $j$,
the assumption that the hypergeometric data correspond to a motive over $\mathbb Q$ leads to a recipe to extend
the definition \eqref{FHS} without the constraint, by only imposing that $q$ is relatively prime to~$M$ in~\eqref{eq-M}.
\end{comment}

% We now follow the recipe of Beukers--Cohen--Mellit \cite{BCM}: from hypergeometric data, we may define finite field hypergeometric sums as follows.  
There exist $p_1, \dots , p_r,q_1, \dots , q_s \in \Z_{\geq 1}$ such that
\begin{equation}
\prod_{j=1}^d \frac{ x - e^{2\pi i\alpha_j}}{ x - e^{2\pi i\beta_j}}=\frac{\prod_{j=1}^r (x^{p_j}-1)} {\prod_{j=1}^s (x^{q_j}-1)},
\end{equation}
and we define
\begin{equation}
\label{eq-M}
M \colonequals \frac{p_1^{p_1} \cdots p_r^{p_r}}{q_1^{q_1} \cdots q_s^{q_s}}.
\end{equation}

\begin{comment}
Let $D(x) \colonequals \gcd(\prod_{j=1}^r (x^{p_j} - 1), \prod_{j=1}^s (x^{q_j} - 1))$, let $\epsilon \colonequals (-1)^{\sum_{j=1}^s q_j}$, and let $s(m) \in \Z_{\geq 0}$ be the multiplicity of the root $e^{2\pi i m /(q-1)}$ in $D(x)$.  

We abbreviate
\begin{equation} \label{eqn:gpmqm}
g(\pmb{p}m,-\pmb{q}m) \colonequals g(p_1m) \cdots g(p_rm) g(-q_1m) \cdots g(-q_sm).
\end{equation}
For $t \in \F_q^\times$, we define the \defi{finite field hypergeometric sum} \cite{BCM15}
\begin{equation} 
H_q(\pmb{\alpha}, \pmb{\beta} \,|\, t) \colonequals \frac{(-1)^{r+s}}{1-q} \sum_{m=0}^{q-2} q^{-s(0) + s(m)} g(\pmb{p}m,-\pmb{q}m)\omega(\epsilon M^{-1}t)^m.
\end{equation}
The details of this extension of the finite hypergeometric sums
can be found in \cite[Theorem~1.3]{BCM15} or \cite[Sect.~4.2]{LTYZ17}.
\end{comment}

Computing the local $L$-factors at good primes is completely automated in the \textsc{Magma} \cite{Magma} package of hypergeometric motives.
% The related implementation can be executed on online Magma Calculator \url{http://magma.maths.usyd.edu.au/calc/}.

\subsection*{Motive and $L$-function}

The finite field hypergeometric sums arose in counting points on algebraic varieties over finite fields, and they combine to give motivic $L$-functions following Beukers--Cohen--Mellit \cite{BCM15}, as follows.  For a parameter $\lambda$, let $V_\lambda$ be the pencil of varieties in weighted projective space defined by the equations
\begin{equation} \label{eqn:xiyj}
x_1+x_2+\dots + x_r = y_1+ \dots + y_s, \quad \lambda x_1^{p_1} \cdots x_r^{p_r} =y_1^{q_1} \cdots y_s^{q_s}
\end{equation}
and subject to $x_i, y_j \ne 0$.  The pencil $V_\lambda$ is affine and singular \cite[Section 5]{BCM15}; in fact, it is smooth outside of $\lambda = 1/M$, where it acquires an ordinary double point.

\begin{theorem}
\label{th-BCM}
Suppose that $\gcd(p_1, \dots , p_r,q_1,\dots, q_s)=1$ and $M\lambda \neq 1$.  Then there exists a
suitable completion $\overline{V_\lambda}$ of $V_\lambda$ such that
$$
\#\overline{V_\lambda}(\mathbb F_q)= P_{rs}(q)+(-1)^{r+s-1}q^{\min(r-1,s-1)}H(\balpha,\bbeta\mid M\lambda),
$$
% where $\overline{V_\lambda}(\mathbb F_q)$ is the set of $\mathbb F_q$-rational points
% on $\overline{V_\lambda}$, the number $M$ is given in~\eqref{eq-M} 
and where $P_{rs}(q) \in \Q(q)$ is explicitly given.
% $$
% P_{rs}(q)=\sum_{m=0}^{\min(r-1,s-1)}\binom{r-1}{m}\binom{s-1}{m}
% \frac{q^{r+s-m-2}-q^m}{q-1}.
% $$
\end{theorem} 	

The completion provided in Theorem \ref{th-BCM} may still be singular, and a nonsingular completion is not  currently known in general; we expect that $\overline{V_\lambda}$ has only quotient singularities, and hence behaves like a smooth manifold with respect to rational cohomology, by the nature of the toric (partial) desingularization.  In any event, this theorem shows that the sums \eqref{FHS} have an explicit connection to arithmetic geometry and complex analysis.  

We accordingly define hypergeometric $L$-functions, as follows.  Let $S_{\lambda}$ be the set of primes dividing the numerator or denominator in $M$ together with the primes dividing the numerator or denominator of $M\lambda$ or $M\lambda-1$.  A prime $p \not\in S_\lambda$ is called \defi{good} (for $\balpha,\bbeta,\lambda$).  For a good prime $p$, we define the formal series
\begin{equation} 
L_p(H(\balpha,\bbeta\,|\,M\lambda),T) \colonequals \exp\left(-\sum_{r=1}^{\infty} H_{p^r}(\balpha,\bbeta\,|\,M\lambda)\frac{T^r}{r}\right) \in 1 + T\Q[[T]].
\end{equation}

\begin{corollary}
For $p \not\in S_\lambda$ and $\lambda \in \F_p^{\times}$, we have
\[ L_p(H(\balpha,\bbeta\,|\,M\lambda),T) \in \Q[T]. \]
\end{corollary}

\begin{proof}
The zeta function of $\overline{V_\lambda}$ over $\F_p$ is a rational function by work of Dwork; the exponential series for $P_{rs}(q)$ is also rational, so the result follows from Theorem \ref{th-BCM}.
\end{proof}

\begin{remark}
In fact, we expect that $L_p(H(\balpha,\bbeta\,|\,M\lambda),T) \in 1+T\Z[T]$ is a polynomial of degree $d$;
it should follow from the construction in Theorem \ref{th-BCM} or from work of Katz \cite{Ka76}, but we could not find a published proof.
We establish this property in the cases we consider, as a byproduct of our analysis.
\end{remark}

Globalizing, we define the incomplete $L$-series
\begin{equation} 
L_S(H(\balpha,\bbeta\,|\,M\lambda),s) = \prod_{p \not\in S} L_p(H(\balpha,\bbeta\,|\,M\lambda),p^{-s})^{-1}
\end{equation}
a Dirichlet series that converges in a right half-plane, but otherwise remains rather mysterious.  Our goal in what follows will be to match such $L$-functions (coming from geometry, rapidly computable) with $L$-functions of modular forms in certain cases, so that the former can be completed to inherit the good properties of the latter.  

\subsection*{Examples}

We conclude this section with two examples.

\begin{example} \label{exm:K3ex}
We return to our motivating example, with the parameters $\balpha = \{ \frac12, \frac14, \frac34\}$ and $\balpha=\{1,1,1\}$, we find $p_1=4$ and $q_1=\cdots=q_4=1$.  Then eliminating $x_1$ in \eqref{eqn:xiyj} gives
\[ V_\lambda \colon \lambda (y_1+y_2+y_3+y_4)^4 = y_1y_2y_3y_4; \]
and Theorem \ref{th-BCM} yields
\begin{equation*}
\#\overline{V_\lambda}(\mathbb F_q)
=\frac{q^3-1}{q-1}+H_q(\tfrac12,\tfrac14,\tfrac34;1,1,1 \,|\, 4^4\lambda).
\end{equation*}

We make a change of parameters $\lambda^{-1}=4^4\psi^4$ and consider the pencil of quartic K3 hypersurfaces with generically smooth fibers defined by
\begin{equation} \label{eqn:Xpsix04}
X_{\psi} \colon x_0^4+x_1^4+x_2^4+x_3^4 = 4\psi x_0x_1x_2x_3
\end{equation}
as in \eqref{eqn:XpsiK3}, with generic Picard rank $19$.  The family \ref{eqn:Xpsix04} is known as the Fermat--Dwork family and is well studied (going back to Dwork \cite[\S 6j, p.\ 73]{padic}; see e.g.\ Doran--Kelly--Salerno--Sperber--Voight--Whitcher \cite[\S 1.5]{DKSSVW18b} for further references).  In the context of mirror symmetry, one realizes $\overline{V_\lambda}$ as the mirror of $X_{\psi}$ \cite[\S 5.2]{DKSSVW18a} in the following way, due to
Batyrev: there is an action of $G=(\Z/4\Z)^3$ on $X_{\psi}$, and $V_{\lambda}$ is birational to $X_{\psi}/G$.  
We see again that the finite field hypergeometric sum $H(\tfrac12,\tfrac14,\tfrac34;1,1,1\,|\,\psi^{-4})$ contributes nontrivially to the point counts \cite[Main Theorem 1.4.1(a)]{DKSSVW18b}.

In either model, the holomorphic periods of $\overline{V_\lambda}$ or $X_{\psi}$ are given by the hypergeometric series
\begin{equation}
F(\tfrac12,\tfrac14,\tfrac34;1,1,1\,|\, \psi^{-4} ) = 
\sum_{n=0}^\infty\frac{(\frac12)_n(\frac14)_n(\frac34)_n}{n!^3}\,(4^4\lambda)^n
=\sum_{n=0}^\infty\frac{(4n)!}{n!^4}\,\lambda^n.
\end{equation}

As mentioned in the introduction, at the specialization $\psi^4=4^4\lambda=-1/48$, the K3 surface is singular, with Picard number $20$---it is this rare event that explains the formula \eqref{rama1}.  Computing the local $L$-factors, we find
\[ L_p(H(\tfrac12,\tfrac14,\tfrac34;1,1,1\,|\,{-1/48}),T) = (1-\chi(p)pT)(1-b_pT+p^2T^2) \]
for $p \neq 2,3$, where $\chi(p)=\displaystyle{\legen{12}{p}}$ is the quadratic character attached to $\Q(\sqrt{12})$ and $b_p \in \Z$ defined in \eqref{ap43}.  Indeed, this factorization agrees with the fact that the global $L$-series can be completed to 
\[ L(H(\tfrac12,\tfrac14,\tfrac34;1,1,1\,|\,{-1/48}),s)=L(f,s,\Sym^2) \]
where $f$ is the classical modular form with LMFDB label \href{http://www.lmfdb.org/ModularForm/GL2/Q/holomorphic/144/2/c/a/}{\textsf{144.2.c.a}}: more generally, see Elkies--Sch\"utt \cite{ES}, Doran--Kelly--Salerno--Sperber--Voight--Whitcher \cite[Theorem 5.1.3]{DKSSVW18a}, or Zudilin \cite[Observation~4]{Zu18}.  Consequently, the completed hypergeometric $L$-series inherits analytic continuation and functional equation.  
\begin{comment}
This way we find out that the local $L$-factors of the motive for \eqref{CSF}, \eqref{rama1} assume the form
$$
(1-\chi_{-3}(p)pT)(1-a(p)T+\chi_{-3}(p)^2\chi_{-4}(p)p^2T^2) \quad\text{for primes} \; p\ne2,
$$
where $a(p)=a_{\text{\eqref{rama1}}}(p)$ are given in \eqref{ap43}.

Similarly, the local $L$-factors for \eqref{rama1b} read
$$
(1-\chi_{-15}(p)pT)(1-a_{\text{\eqref{rama1b}}}(p)T+\chi_{-15}(p)^2\chi_{-43}(p)p^2T^2) \quad\text{for primes} \; p\ne43,
$$
where
$$
a_{\text{\eqref{rama1b}}}(p)=\begin{cases}
\bigl(\frac{-15}p\bigr)(x^2-2p) & \text{when}\; p=(x^2+43y^2)/4, \\
0 & \text{if}\; \bigl(\frac{-43}p\bigr)=-1.
\end{cases}
$$
\end{comment}
\end{example}

\begin{example}
We consider the hypergeometric data attached to Ramanujan-type formula~\eqref{rama2b}, corresponding to $\#13$ in Table \ref{tab1} and with parameters $\balpha=\{\tfrac12,\tfrac13,\tfrac23,\tfrac16,\tfrac56\}$ and $\bbeta=\{1,\dots,1\}$.  This example is, in many aspects, runs parallel to Example \ref{exm:K3ex} and the related mirror symmetry construction of the famous quintic threefold \cite{COGP91}.  We have
\[ V_\lambda \colon \lambda(y_1+y_2+\dots+y_6)^6= y_1y_2\cdots y_6 \]
and Theorem~\ref{th-BCM} implies
\begin{equation*}
\#\overline{V_\lambda}(\mathbb F_q)
=\frac{q^5-1}{q-1}+H(\tfrac12,\tfrac13,\tfrac23,\tfrac16,\tfrac56;1,1,1,1,1 \,|\, 6^6\lambda).
\end{equation*}

Alternatively, we consider the pencil of sextic fourfolds
\[ 
X_\psi\colon x_0^6+x_1^6+x_2^6+x_3^6+x_4^6+x_5^6=6\psi x_0x_1x_2x_3x_4x_5 \]
in $\mathbb P^5$.  Under the change of parameter $\lambda^{-1}=6^6 \psi^6$, we find that $V_\lambda$ is birational to $X_\psi/G$ where $G \simeq (\Z/6\Z)^5$.  
\begin{comment}
For fixed $\psi$, the equation admits the action of the discrete group
$$
G=\{(\zeta_6^{a_1},\dots,\zeta_6^{a_6}):a_1+\dots + a_6\equiv 0 \pmod 6\}\cong (\mathbb Z/6\mathbb Z)^5
$$
via the map $(X_1,\dots, X_6)\mapsto(\zeta_6^{a_1}X_1,\dots,\zeta_6^{a_6}X_6)$,
where $\zeta_6=e^{\pi i/3}$ is the primitive $6$-th root of unity.
Following the recipe of Batyrev, the mirror fourfold is constructed from the orbifold $\mathcal X_\psi/G$.
One way to realize the quotient is letting $y_j=X_j^6$ for $j=1,\dots,6$, $x_1=6\psi X_1\cdots X_6$ and $\lambda=(6\psi)^{-6}$,
\end{comment}
The $X_{\psi}$ are generically Calabi--Yau fourfolds.  A computation (analogous to Candelas--de la Ossa--Greene--Parks \cite{COGP91}) shows that the Picard--Fuchs differential operator is given by
$$
\biggl(\lambda\frac{\d}{\d\lambda}\biggr)^5
-6\lambda\biggl(6\lambda\frac{\d}{\d\lambda}+1\biggr)\biggl(6\lambda\frac{\d}{\d\lambda}+2\biggr)\biggl(6\lambda\frac{\d}{\d\lambda}+3\biggr)
\biggl(6\lambda\frac{\d}{\d\lambda}+4\biggr)\biggl(6\lambda\frac{\d}{\d\lambda}+5\biggr).
$$
The unique (up to scalar) holomorphic solution near zero is the hypergeometric function
$$
F(\balpha,\bbeta\,|\,6^6\lambda) = \sum_{n=0}^\infty\frac{(\frac12)_n(\frac13)_n(\frac23)_n(\frac16)_n(\frac56)_n}{n!^5}\,(6^6\lambda)^n
=\sum_{n=0}^\infty\frac{(6n)!}{n!^6}\,\lambda^n.
$$
% Resolving singularities as in \cite{BCM15}, one gets a Calabi--Yau fourfold $\overline{V_\lambda}$, which is the mirror manifold of $\mathcal X_\psi$.

Using the \textsc{Magma} implementation, we compute the first few (good) $L$-factors:
\begin{equation}
\begin{aligned}
L_7(T) &= (1+7^2 T)(1-7^4T^2)(1-35T+7^4T^2) \\
L_{11}(T) &= (1-11^2 T)(1-11^4 T^2)(1 - 30T + 11^4 T^2) \\
L_{13}(T) &= (1+13^2 T)(1 - 64T - 1758\cdot 13 T^2 - 64\cdot 13^4 T^3 + 13^8 T^4)
\end{aligned}
\end{equation}
and observe that for $p \neq 2,3,5$, 
\begin{equation} 
L_p(T) \overset?= (1-\chi_5(p)p^2 T)(1-a_p T + b_p p T - \chi_{129}(p) a_p p^4 T^3 + \chi_{129}(p)p^8 T^4).
\end{equation}
Moreover, when $\chi_{129}(p)=-1$ then $b_p=0$ and the quartic polynomial factors as
\[ 1-a_p T + b_p p T - \chi_{129}(p) a_p p^4 T^3 + \chi_{129}(p)p^8 T^4 \overset?= (1-p^4T^2)(1- a_p T + p^4T^2) \]
whereas for $\chi_{129}(p)=1$ it is generically irreducible.  This suggests again a rare event which we seek to explain using modular forms.  
\end{example}

\section{The Asai transfer of a Hilbert modular form}
\label{sec5}

Having defined $L$-functions arising from hypergeometric motives in the previous sections, over the next two sections we follow the predictions of the Langlands philosophy and seek to identify these $L$-functions as coming from modular forms in the cases of interest.  More precisely, we confirm experimentally a match with the Asai transfer of certain Hilbert modular forms over quadratic fields.  We begin in this section by setting up the needed notation and background.  As general references for Hilbert modular forms, consult Freitag \cite{Freitag} or van der Geer \cite{geer}; for a computational take, see Demb\'el\'e--Voight \cite{DV13}.

Let $F=\Q(\sqrt{d})$ be a real quadratic field of discriminant $d>0$ with ring of integers $\Z_F$ and Galois group $\Gal(F\,|\,\Q)=\langle \tau \rangle$.  By a \defi{prime} of $\Z_F$ we mean a nonzero prime ideal $\frakp \subseteq \Z_F$.  Let $v_1,v_2\colon F \hookrightarrow \R$ be the two embeddings of $F$ into $\R$.  For $x \in F$ we write $x_i \colonequals v_i(x)$, and for $\gamma \in \M_2(F)$ we write $\gamma_i = v_i(\gamma)$ for the coordinate-wise application of $v_i$.  An element $a \in F^\times$ is \defi{totally positive} if $v_1(a),v_2(a)>0$; we write $F_{>0}^\times$ for the group of totally positive elements.  The group
\begin{equation} 
\GL_2^+(F) \colonequals \{\gamma \in \GL_2(F) : \det \gamma \in F_{>0}^\times\} 
\end{equation}
acts on the product $\mathcal{H} \times \mathcal{H}$ of upper half-planes by embedding-wise linear fractional transformations $\gamma(z)\colonequals (\gamma_1(z_1), \gamma_2(z_2))$.

Let $k_1,k_2 \in 2\Z_{>0}$, write $k \colonequals (k_1,k_2)$, and let $k_0 \colonequals \max(k_1,k_2)$ and $w_0 \colonequals k_0-1$.  Let $\frakN \subseteq \Z_F$ be a nonzero ideal.  Let $S_k(\frakN;\psi)$ denote the (finite-dimensional) $\C$-vector space of Hilbert cusp forms of weight $k$, level $\Gamma_0(\frakN)$, and central character $\psi$.  Hilbert cusp forms are the analogue of classical cusp forms, but over the real quadratic field $F$.  When the narrow class number of $F$ is equal to $1$ (i.e., every nonzero ideal of $\Z_F$ is principal, generated by a totally positive element) and $\psi$ is the trivial character, a Hilbert cusp form $f \in S_k(\frakN)$ is a holomorphic function $f \colon \mathcal{H} \times \mathcal{H} \to \C$, vanishing at infinity, such that
\begin{equation} 
f(\gamma z) = (c_1z_1+d_1)^{k_1/2}(c_2z_2+d_2)^{k_2/2} f(z) 
\end{equation}
for all $\gamma=\begin{pmatrix} a & b \\ c & d \end{pmatrix} \in \GL_2^+(\Z_F)$ such that $c \in \frakN$.  

The space $S_k(\frakN)$ is equipped with an action of pairwise commuting Hecke operators $T_\frakp$ indexed by nonzero primes $\frakp \nmid \frakN$.  A Hilbert cusp form $f$ is a \defi{newform} if $f$ is an eigenform for all Hecke operators and $f$ does not arise from $S_k(\frakM)$ with $\frakM \mid \frakN$ a proper divisor.  

Let $f \in S_k(\frakN)$ be a newform.  For $\frakp \nmid \frakN$, we have $T_\frakp f = a_\frakp f$ with $a_\frakp \in \C$ a totally real algebraic integer (the \defi{Hecke eigenvalue}), and we factor
\[ 1-a_\frakp T + \Nm(\frakp)^{w_0} T^2 = (1-\alpha_\frakp T)(1-\beta_\frakp T) \in \C[T] \]
where $\Nm(\frakp)$ is the absolute norm.  Then $\abs{\alpha_\frakp}=\abs{\beta_\frakp}=\sqrt{p}^{w_0}$.  

For $p \in \Z$ prime with $p \nmid \Nm(\frakN)$, following Asai \cite{Asa77} we define, abbreviating $\frakp'=\tau(\frakp)$,
\begin{equation}\label{eqn:Asai-lf}
\begin{aligned}
&L_p(f,T,\As)
%\colonequals
\\
&\qquad
\colonequals
\begin{cases}
(1 - \alpha_\frakp \alpha_{\frakp'}T)(1 - \alpha_{\frakp'}\beta_\frakp T)(1 - \alpha_\frakp \beta_{\frakp'}T)(1 - \beta_\frakp \beta_{\frakp'}T), & \text{if $p\Z_F = \frakp \frakp'$ splits}; \\
(1 - \alpha_\frakp T)(1 - \beta_\frakp T)(1 - \psi(\frakp)p^{2w_0} T^2), & \text{if $p\Z_F=\frakp$ is inert}; \\
(1 - \alpha_\frakp^2 T)(1 - \beta_\frakp^2 T)(1 - \psi(\frakp)p^{w_0} T), & \text{if $p\Z_F = \frakp^2$ ramifies}.
\end{cases}
\end{aligned}
\end{equation}
We call the factors $L_p(f,T,\As)$ the \defi{good} $L$-factors of $f$.  The \defi{partial Asai $L$-function} of $f$ is the Dirichlet series defined by the Euler product
\begin{equation} \label{eqn:LSfs}
L_S(f,s,\As) \colonequals \prod_{p \not\in S} L_p(f,p^{-s},\As)^{-1}
\end{equation}
where $S=\{p : p \mid \Nm(\frakN)\}$.

The key input we need is the following theorem.  For a newform $f \in S_k(\frakN,\psi)$, let $\tau(f)$ be the newform of weight $(k_2,k_1)$ and level $\tau(\frakN)$ with $T_\frakp \tau(f) = a_{\tau(\frakp)} \tau(f)$, with central character $\psi \circ \tau$.  Finally, for central character $\psi$ (of the idele class group of $F$) let $\psi_0$ denote its restriction (to the ideles of $\Q$).  % : on ideals, we have $\psi_0(n\Z) \colonequals \psi( n\Z_F)$.

\begin{theorem}[{Krishnamurty \cite{Kr03}, Ramakrishnan \cite{Ra02}}] \label{thm:LfcnAsai}
Let $f \in S_k(\frakN,\psi)$ be a Hilbert newform, and suppose that $\tau(f)$ is not a twist of $f$.  Then the partial $L$-function $L_S(f,s,\As)$ can be completed to a $\Q$-automorphic $L$-function 
\[ \Lambda(f,s,\As) = N^{s/2} \Gamma_\C(s)^2 L(f,s,\As) \]
of degree $4$, conductor $N \in \Z_{>0}$, with central character $\psi_0^2$.  

More precisely, there exists a cuspidal automorphic representation $\Pi=\Pi_\infty \otimes (\bigotimes_p \Pi_p)$ of $\GL_4(\mathbb{A}_\Q)$ such that $L_p(f,p^{-s},\Asai)=L(s,\Pi_p)^{-1}$ for all $p \nmid \Nm(\frakN)$.  In particular, $L(f,s,\As)$ is entire and satisfies a functional equation $\Lambda(s) = \eps \overline{\Lambda}(8-s)$ with $\abs{\eps}=1$. 
\end{theorem}

The automorphic representation $\Pi$ in Theorem \ref{thm:LfcnAsai} goes by the name \defi{Asai transfer}, \defi{Asai lift}, or \defi{tensor induction} of the automorphic representation $\pi$ attached to $f$, and we write $\Pi=\As(\pi)$.

\begin{proof}
We may identify $\GL_2(\A_F) \cong \Res_{F|\Q} \GL_2(\A_\Q)$, with $L$-group
\begin{equation}
{}^L(\Res_{F|\Q}\GL_2) \cong \GL_2(\C) \times \GL_2(\C)\rtimes \Gal_\Q,
\end{equation}
where $\Gal_\Q \colonequals \Gal(\Qbar\,|\,\Q)$.  We define the $4$-dimensional representation 
\begin{equation}
\begin{aligned}
r \colon \GL_2(\C) \times \GL_2(\C)\rtimes \Gal_\Q &\to \GL(\C^2 \otimes \C^2) \simeq \GL_4(\C) \\
r(g_1, g_2, \sigma) &=
\begin{cases} 
g_1 \otimes g_2, & \text{ if $\sigma|_F = \id$;} \\
g_2 \otimes g_1, & \text{ if $\sigma|_F = \tau$.}
\end{cases}
\end{aligned}
\end{equation}
For a place $v$ of $\Q$, let $r_v$ be the restriction of $r$ to $\GL_2(\C) \times \GL_2(\C) \rtimes \Gal_{\Q_v}$.  

Let $\pi= \pi_{\infty} \otimes (\bigotimes_p \pi_{p})$ be the cuspidal automorphic representation of $\GL_2(\A_F)$ attached to $f$.  Then $\pi_{p}$ is an admissible representation of $\GL_2(F \otimes \Q_p)$ corresponding to an $L$-parameter
\begin{equation} 
\phi_{p} \colon W_p' \to \GL_2(\C) \times \GL_2(\C)\rtimes \Gal_{\Q_p}, 
\end{equation}
where $W_p'$ is the Weil--Deligne group of $\Q_p$.  We define $\As(\pi_{p})$ to be the irreducible admissible representation of $\GL_4(\Q_p)$ attached 
to $r_p \circ \phi_p$ by the local Langlands correspondence, and we combine these to
\begin{equation}
\As(\pi) \colonequals \As(\pi_{\infty}) \otimes \bigotimes_p \As(\pi_{p}).
\end{equation}

By a theorem of Ramakrishnan \cite[Theorem D]{Ra02} or Krishnamurty \cite[Theorem 6.7]{Kr03}, $\As(\pi)$ is an automorphic representation of $\GL_4(\A_\Q)$ whose $L$-function is defined by
\begin{equation}
L(s, \pi, \As) \colonequals L(s, \pi_\infty, r_\infty \circ \phi_\infty)\prod_{p} L(s, \pi_p, r_p \circ \phi_p)
\end{equation}
whose good $L$-factors agree with \eqref{eqn:Asai-lf} \cite[\S 4]{Kr03}.  Under the hypothesis that $\tau(f)$ is not a twist of $f$, we conclude that $\As(\pi)$ is cuspidal \cite[Theorem D(b)]{Ra02}. 
% In particular, when $p$ is unramified and $\pi_{p}$ is spherical, we have 
% $$L(s, \pi_p, r_p^{\pm}) = \det(1 - r^{\pm}(A(\pi_p))p^{-s})^{-1},$$
% where $A(\pi_p)$ is the Satake parameter of $\phi_p$ in ${}^L(\Res_{F/\Q}\GL_2)(\C)$, i.e. the semisimple
% conjugacy class of $r_p^{\pm}\circ \phi_p$. 
Consequently, we may take $\Pi=\As(\pi)$ in the theorem.
\end{proof}

\begin{remark}
Some authors also define the representation $\As^-(\pi)$, which is the quadratic twist of $\As(\pi)$ by the quadratic character attached to $F$.
\end{remark}

In addition to the direct construction \eqref{eqn:LSfs} and the automorphic realization in Theorem \ref{thm:LfcnAsai}, one can also realize the Asai $L$-function via Galois representations.  By Taylor \cite[Theorem 1.2]{Tay89}, attached to $f$ is a Galois representation
\[ \rho \colon \Gal_F \to \GL(V) \simeq \GL_2(\Qlbar) \]
such that for each prime $\frakp \nmid \frakN$, we have
\[ \det(1-\rho(\Frob_\frakp)T) = 1 - a_\frakp T + \Nm(\frakp)^{w_0} T^2. \]
Then there is a natural extension of $\rho$ to $\Gal_\Q$, a special case of \defi{multiplicative induction} (or \defi{tensor induction}) \cite[\S 7]{Pr92} defined as follows: for a lift of $\tau$ to $\Gal_\Q$ which by abuse is also denoted $\tau$, we define \cite[p.~1363]{Kr02a} (taking a left action)
\begin{equation}
\begin{aligned}
\As(\rho) \colon \Gal_\Q &\to \GL(V \otimes V) \simeq \GL_4(\Qlbar) \\
\As(\rho)(\sigma)(x \otimes y) &= 
\begin{cases}
\rho(\sigma)(x) \otimes \rho(\tau^{-1}\sigma \tau)(y), & \text{if $\sigma|_F=\id$;} \\
\rho(\sigma\tau)(y) \otimes \rho(\tau^{-1}\sigma)(x), & \text{if $\sigma|_F=\tau|_F$.} 
\end{cases}
\end{aligned}
\end{equation}
Up to isomorphism, this representation does not depend on the choice of lift $\tau$.  A direct computation \cite[Lemma 3.3.1]{Kr02a} then verifies that $\det(1-\As(\rho)(\Frob_p)T)=L_p(f,T,\As)$ as defined in \eqref{eqn:LSfs}.

The bad $L$-factors $L_p(f,T,\Asai)$ and conductor $N$ of $L(f,s,\Asai)$ are uniquely determined by the good $L$-factors, but they are not always straightforward to compute. 

\section{Matching the hypergeometric and Asai $L$-functions}
 \label{sec6}
 
%In this section, we
We now turn to the main conjecture of this paper.

\subsection*{Main conjecture}

We propose the following conjecture.

\begin{conjecture} \label{mainconj}
Let $\balpha$ be a set of parameters from Table~\textup{\ref{tab1}} and let $\bbeta=\{1,1,1,1,1\}$.  
Then there exist quadratic Dirichlet characters $\chi,\varepsilon$ and a Hilbert cusp form $f$ over a real quadratic field $F$ of weight $(2,4)$ such that for all good primes $p$ we have
\[ L_p(H(\balpha,\bbeta\,|\,z), T) \overset?= (1-\chi(p)p^2T)L_p(f,T/p,\Asai,\varepsilon). \]
In particular, we have the identity
\[ L(H(\balpha,\bbeta\,|\,z),s) \overset?= L(s-2,\chi)L(f,s+1,\Asai,\varepsilon). \]
\end{conjecture}

We can be more precise in Conjecture \ref{mainconj} for some of the rows, as follows.  Let $\#n$ be a row in Table~\textup{\ref{tab1}} with $n \neq 7,13,14,15,16$.  Then we conjecture that the central character $\psi$ of $f$ is a quadratic character of the class group of $F$ induced from a Dirichlet character; and the conductors of $\chi,\varepsilon,\psi$, the discriminant $d_F$ of $F$, and the level $\frakN$ of $f$ are indicated in Table~\textup{\ref{tab2}}.

\begin{equation} \label{tab2}\addtocounter{equation}{1} \notag
\begin{gathered}
{\renewcommand{\arraystretch}{1.2}
\begin{tabular}{r|cc|ccccc} 
\# & $\balpha$  & $z$ & $\chi$ & $d_F$ & $\mathfrak{N}$ & $\psi$ & $\varepsilon$ \\
\hline
\hline
 1, 2 & $\{\tfrac12,\tfrac12,\tfrac12,\tfrac12,\tfrac12\}$ & $-2^{\pm 2} \vphantom{|^{0^0}}$ & $1 \vphantom{|^{0^0}}$ & $5$ & $(4)$ & $1$ & $1$ \\
 3, 4 & $\{\tfrac12,\tfrac12,\tfrac12,\tfrac12,\tfrac12\}$ & $-2^{\pm 10}$ & $1$ & $41$ & $(1)$ & $1$ & $1$ \\
 5 & $\{\tfrac12,\tfrac12,\tfrac12,\tfrac13,\tfrac23\}$ & $(3/4)^3$ & $1$ & $37$ & $(1)$ & $1$ & $1$ \\
 6 & $\{\tfrac12,\tfrac12,\tfrac12,\tfrac13,\tfrac23\}$ & $-3^3$ & $1$ & $28$ & $(8)$ & $1$ & $-4$ \\
 7 & $\{\tfrac12,\tfrac15,\tfrac25,\tfrac35,\tfrac45\}$ & $-5^5/2^8$ & $1?$ & $69$ & ? & ? & $1?$ \\
 8 & $\{\tfrac12,\tfrac12,\tfrac12,\tfrac14,\tfrac34\}$ & $1/2^4$ & $1$ & $60$ & $(4)$ & $3$ & $1$ \\
 9 & $\{\tfrac12,\tfrac13,\tfrac23,\tfrac14,\tfrac34\}$ & $-1/48$ & $1$ & $12$ & $(81)$ & $1$ & $1$ \\
10 & $\{\tfrac12,\tfrac13,\tfrac23,\tfrac14,\tfrac34\}$ & $-3^3/2^4$ & $1$ & $172$ & $(4)$ & $1$ & $1$ \\
11 & $\{\tfrac12,\tfrac13,\tfrac23,\tfrac16,\tfrac56\}$ & $-(3/4)^6$ & $1$ & $193$ & $(1)$ & $1$ & $1$ \\
12 & $\{\tfrac12,\tfrac13,\tfrac23,\tfrac16,\tfrac56\}$ & $(3/5)^6$ & $1$ & $76$ & $(8)$ & $1$ & $-4$ \\
13 & $\{\tfrac12,\tfrac13,\tfrac23,\tfrac16,\tfrac56\}$ & $-1/80^3$ & $5?$ & $129$ & ? & ? & $1?$ \\
14 & $\{\tfrac12,\tfrac14,\tfrac34,\tfrac16,\tfrac56\}$ & $-1/2^{10}$ & $12?$ & $492$ & ? & ? & $1?$ \\
15 & $\{\tfrac12,\tfrac18,\tfrac38,\tfrac58,\tfrac78\}$ & $1/7^4$ & $28?$ & $168$ & ? & ? & $-7?$ \\
\end{tabular}} \\
\text{Table \ref{tab2}: Hilbert modular form data}
\end{gathered}
\end{equation}

\subsection*{Evidence}

We verified Conjecture \ref{mainconj} for the rows indicated in Table \ref{tab2} using \textsc{Magma} \cite{Magma}; the algorithms for hypergeometric motives were implemented by Watkins, algorithms for $L$-functions implemented by Tim Dokchitser, and algorithms for Hilbert modular forms by Demb\'el\'e, Donnelly, Kirschmer, and Voight.  The code is available online \cite{codeonline}.

Moreover, using the $L$-factor data in Table \ref{tab3}, we have confirmed the functional equation for $L(H(\balpha,\bbeta\,|\,z),s)$ up to 20 decimal digits for all but \#13.  When the discriminant $d_F$ and the level $\frakN$ are coprime, we observe that the conductor of $L(f,s,\Asai)$ is  $N=d_F\Nm(\frakN)$.

\begin{equation} \label{tab3}\addtocounter{equation}{1} \notag
\begin{gathered}
{\renewcommand{\arraystretch}{1.1}
\begin{tabular}{c|c|ccc}
\# & $N$ & $p$ & $\ord_p(N)$ & $L_p(f,T,\Asai,\varepsilon)$ \\
\hline
\hline
 \multirow{2}{*}{1, 2} & \multirow{2}{*}{$80$} &  $2$ & $4$ & $1$ \\
& & $5$ & $1$ & $(1-p^2T)(1+6pT+p^4T^2)$ \\
 \hline
 \multirow{2}{*}{3, 4} & \multirow{2}{*}{$41$} & $2$ & $0$ & $1+5T+5pT^2+5p^4T^3+p^8T^4$ \\
 & & $41$ & $1$ & $(1-p^2T)(1-18pT-p^4T^2)$ \\
 \hline
 \multirow{3}{*}{5} & \multirow{3}{*}{$37$} & $2$ & $0$ & $(1-p^2T)^2(1+3pT+p^4T^2)$ \\
   & & $3$ & $0$ & $1+11T+28pT^2+11p^4T^3+p^8T^4$ \\
   & & $37$ & $1$ & $(1-p^2T)(1+70pT+p^4T^2)$ \\
   \hline
 \multirow{3}{*}{6} & \multirow{3}{*}{$112$} & $2$ & $4$ & $1-p^4T^2$ \\
   & & $3$ & $0$ & $1+8T+10pT^2+8p^4T^3+p^8T^4$ \\
   & & $7$ & $1$ & $(1+p^2T)(1+46T+p^4T^2)$ \\
   \hline
 \multirow{3}{*}{7} & \multirow{3}{*}{$69$} & $2$ & $0$ & $(1-p^4T^2)(1+p^4T^2)$ \\
   & & $3$ & $1$ & $(1+p^2T)(1+5T+p^4T^2)$ \\
   & & $5$ & $0$ & $1+4T-14pT^2+4p^4T^3+p^8T^4$ \\
   & & $23$ & $1$ & $(1+p^2T)(1-470T+23^4T^2)$ \\
   \hline   
 \multirow{3}{*}{8} & \multirow{3}{*}{$60$} & $2$ & $2$ & $(1-p^2T)(1+3pT+p^4T^2)$ \\
   & & $3$ & $1$ & $(1+p^2T)(1+2T+p^4T^2)$ \\
   & & $5$ & $1$ & $(1+p^2T)(1-2T+p^4T^2)$ \\
   \hline
 \multirow{2}{*}{9} & \multirow{2}{*}{$972$} & $2$ & $2$ & $(1-p^2T)(1-p^4T^2)$ \\
  & & $3$ & $5$ & $1-p^4T^2$ \\ 
  \hline
 \multirow{3}{*}{10} & \multirow{3}{*}{$172$} & $2$ & $2$ & $(1-p^2T)(1-p^4T^2)$ \\
  & & $3$ & $0$ & $1+14T+34pT^2+14p^4T^3+p^8T^4$ \\
  & & $43$ & $1$ & $(1-p^2T)(1+22pT+p^4T^2)$ \\
  \hline
 \multirow{2}{*}{11} & \multirow{2}{*}{$193$} & $2$ & $0$ & $(1-p^4T^2)^2$ \\
 & & $193$ & $1$ & $(1-p^2T)(1+361pT+p^4T^2)$ \\
 \hline
 \multirow{4}{*}{12} & \multirow{4}{*}{$304$} & $2$ & $4$ & $1-p^4T^2$ \\
   & & $3$ & $0$ & $1+5T-8pT^2+5p^4T^3+p^8T^4$ \\
   & & $5$ & $0$ & $1-250pT^2+p^8T^4$ \\
   & & $19$ & $1$ & $(1+p^2T)(1+178T+p^4T^2)$ \\
 \hline
 \multirow{3}{*}{14} & \multirow{3}{*}{$850176$} & $2$ & $8$ & $1+p^2T$ \\
   & & $3$ & $4$ & $1-p^2T$ \\
   & & $41$ & $1$ & $(1+p^2T)(1-32pT^2+p^4T^2)$ \\
 \hline
 \multirow{3}{*}{15} & \multirow{3}{*}{$59006976$} & $2$ & $13$ & $1$ \\
   & & $3$ & $1$ & $(1+p^2T)(1-4T+p^4T^2)$ \\
   & & $7$ & $4$ & $1+p^2T$ \\
\end{tabular}} \\
\text{Table \ref{tab3}: $L$-factor data for $L_p(H(\balpha,\bbeta\,|\,z),T) \overset?=L_p(\chi,T)L_p(f,T,\Asai,\varepsilon)$}
\end{gathered}
\end{equation}

\begin{remark}
\label{new-Guillera}
In a recent arithmetic study of his formulas for $1/\pi^2$, Guillera \cite{Gu20} comes up with an explicit recipe to cook up the two quadratic characters for each such formula.
He calls them $\chi_0$ and $\varepsilon_0$ and records them in \cite[Table~3]{Gu20}.
Quite surprisingly, they coincide with our $\chi$ and $\varepsilon$ in Table~\ref{tab2}.
\end{remark}

\begin{example}
Consider row $\#1$.  In the space $S_{(2,4)}(4)$ of Hilbert cusp forms over $F=\Q(\sqrt{5})$ of weight $(2,4)$ and level $(4)$ with trivial central character, we find a unique newform $f$ with first few Hecke eigenvalues $a_{(2)}=0$, $a_{(3)}=-30$, $a_{(\sqrt{5})}=-10$, $a_{(7)}=-70$, and $a_{\frakp},a_{\tau(\frakp)}=12 \pm 8\sqrt{5}$, giving for example 
\[ L_3(f,T,\Asai)=(1-3^6T^2)(1+10\cdot 3T+3^6T^2); \]
we then match
\[ L_3(H(\tfrac12,\dots,\tfrac12;1,\dots,1\,|\,{-1/2^2}),T) = (1-3^2T)L_3(f,T/3,\Asai). \]
We matched $L$-factors for all good primes $p$ such that a prime $\frakp$ of $F$ lying over $p$ has $\Nm(\frakp) \leq 200$.  
\end{example}

\begin{example}
For row $\#9$, the space of Hilbert cusp forms over $F=\Q(\sqrt{12})$ of weight $(2,4)$ and level $\frakN=(81)$ has dimension $2186$ with a newspace of dimension $972$.  We find a form $f$ with Hecke eigenvalues $a_{(5)}=140$, $a_{(7)}=98$, \dots; accordingly, we find
\begin{equation} 
\begin{aligned}
L_5(H(\balpha,\bbeta\,|\,{-1/48}),T) &= (1-5^2T)L_5(f,T/5,\Asai) \\
&= (1-5^2 T)^2(1+5^2T)(1+28T+625T^2), 
\end{aligned}
\end{equation}
and so on.  We again matched Hecke eigenvalues up to prime norm $200$.
\end{example}

\begin{remark}
To match row $\#16$ in Table \ref{tab1} with a candidate Hilbert modular form, we would need to extend the implementation of hypergeometric motives to apply for specialization at points $z \not\in \Q$; we expect this extension to be straightforward, given the current implementation of finite field hypergeometric sums.  

By contrast, to match the final rows $\#7$ and $\#13$--$\#15$, we run into difficulty with computing spaces of Hilbert modular forms: we looked for forms in low level, but the dimensions grow too quickly with the level.  We also currently lack the ability to efficiently compute with arbitrary nontrivial central character.  We plan to return to these examples with a new approach to computing systems of Hecke eigenvalues for Hilbert modular forms in future work.
\end{remark}

\begin{remark}
\label{par-structure}
Returning to Remark \ref{par-structure0}, we observe structure in the specialization points $z$ from Table~\ref{tab1}: beyond patterns in the factorization of $z$ and $1-z$, we also note that for these points the completed $L$-function typically has unusually small conductor~$N$, as in Table \ref{tab3}.  (Perhaps a twist of \#15 has smaller conductor?)  Some general observations that may explain this conductor drop:
\begin{itemize}
\item Factor $N = N_1N_2$ where $N_1$ consists of the product of primes $p \mid N$ that divide the least common denominator of $\balpha$ or the numerator or denominator of $z$.  Then $N_2$ should be the squarefree part of the numerator of $1-z$; this numerator is divisible by a nontrivial square in ten of the fifteen cases.
\item The power of $p$ dividing the numerator or denominator of $z$ is itself a multiple of $p$ for most primes $p$ dividing a denominator in $\balpha$.  
\item For a prime $p$, define $s_p(\alpha) = 0$ if $\alpha$ is coprime to $p$ and otherwise let $s_p(\alpha) = \ord_p(\alpha)+1/(p-1)$.  If $\ord_p(z)$ is a multiple of $\sum_{j=1}^5s_p(\alpha_j)$, then $\ord_p(N)$ tends to be especially small.
\end{itemize}
These last two phenomena were first observed by Rodriguez--Villegas; we thank the referee for these observations.

While not making any assertions about completeness, these observations give some indication of why our Table~\ref{tab1} is so short: the specialization points $z$ like those listed are quite rare, and they seem to depend on a pleasing but remarkable arithmetic confluence.  It would be certainly valuable to be able to predict more generally and precisely the conductor of hypergeometric $L$-functions.
\end{remark}

\subsection*{Method}

We now discuss the recipe by which we found a match.  For simplicity, we exclude the case $\#8$ and suppose that the central character $\psi$ is trivial.  In a nutshell, our method uses good split ordinary primes to recover the Hecke eigenvalues up to sign.  

We start with the hypergeometric motive and compute $L_p(H,T) \colonequals L_p(H(\balpha,\bbeta\,|\,z),T)$ for many good primes $p$.  We first guess $\chi$ and $d_F$ by factoring $L_p(H,T)=(1-\chi(p)p^2 T)Q_p(T)$: for primes $p$ that are split in $F$, we usually have $Q_p(T)$ irreducible whereas and for inert primes we find $(1-p^4 T^2) \mid Q_p(T)$.  We observe in many cases that $d_F$ is (up to squares) the numerator of $1-z$.  Combining this information gives us a good guess for $\chi$ and $d_F$.  

We now try to guess the Hecke eigenvalues of a candidate Hilbert newform $f$ of weight $(2,4)$.  Let $p=\frakp \tau(\frakp)$ be a good split prime, and suppose that $p$ is \defi{ordinary} for $f$, i.e., the normalized valuations $\ord_p(a_\frakp),\ord_p(a_{\tau(\frakp)})=0,1$ are as small as possible, or equivalently, factoring
\begin{equation}
\begin{aligned}
L_\frakp(f,T) &= 1-a_\frakp T + p^3T^2=(1-\alpha_\frakp T)(1-\beta_\frakp T)  \\
L_{\tau(\frakp)}(f,T) &= 1-a_{\tau(\frakp)} T + p^3T^2=(1-\alpha_{\tau(\frakp)} T)(1-\beta_{\tau(\frakp)}T)  
\end{aligned}
\end{equation}
we may choose $\frakp$ so that $\alpha_\frakp,\alpha_{\tau(\frakp)}/p$ are $p$-adic units.  We expect that such primes will be abundant, though that seems difficult to prove.  Then $L_p(f,T,\Asai)$ has Hodge--Tate weights (i.e., reciprocal roots with valuations) $(0,3) \otimes (1,2)=(1,2,4,5)$ (adding pairwise) so the Tate twist $L_p(f,T/p,\Asai)$ has Hodge--Tate weights $(0,1,3,4)$ and coefficients with valuations $0,0,1,3,4,8$, matching that of the hypergeometric motive.  

%  are $(0,4)$ and $(1,3)$, so the $p$-adic Hodge--Tate weights will be $(0,1,3,4)$ (adding pairwise).  We say that $f$ is \defi{ordinary} at $p$ if the reciprocal roots of $L_p(f,T,\Asai)$ have normalized valuations precisely $0,1,3,4$; we expect that ordinary primes will be abundant.  

% Let $N$ be the conductor of $M$, and $p \nmid N$ a prime.
% We recall that the Hodge numbers of $M$ are $(0,4)$ and $(1,3)$. So the $p$-adic Hodge-Tate weights of $M$ 
% are $0, 1, 3, 4$. We say that $M$ is {\it ordinary} at $p$ if we can write the Euler factor at $p$ as
% $$Q_p'(T) = (1 - \delta_0 T)(1 - \delta_1 T)(1 - \delta_2 T)(1 - \delta_3 T),$$ such that
% $v_p(\delta_0) = 0$, $v_p(\delta_1) = 1$, $v_p(\delta_2) = 3$, $v_p(\delta_3) = 4$,
% and $\delta_0, \delta_1/p, \delta_2/p^3, \delta_3/p^4$ are distinct. So the eigenvalues 
% $\delta_i$, $i = 0, 1, 2, 3$ are uniquely determined by their valuations. 

So we factor $Q_p(T)$ over the $p$-adic numbers, identifying ordinary $p$ when the roots $\delta_0,\delta_1,\delta_3,\delta_4$ have corresponding valuations $0,1,3,4$.  Then we have the equations 
\begin{equation} 
\begin{aligned}
p\delta_0 &= \alpha_\frakp \alpha_{\tau(\frakp)} \\
p\delta_1 &= \alpha_\frakp \beta_{\tau(\frakp)}
\end{aligned}
\end{equation}
and two similar equations for $\delta_3,\delta_4$.  Therefore
\begin{equation} 
p^2\delta_0\delta_1 = \alpha_\frakp^2 \alpha_{\tau(\frakp)}\beta_{\tau(\frakp)} = \alpha_\frakp^2 p^3 
\end{equation}
so 
\begin{equation} 
\alpha_\frakp = \pm \sqrt{\frac{\delta_0\delta_1}{p}};
\end{equation}
and this determines the Hecke eigenvalue
\begin{equation} \label{eqn:valueofap}
a_\frakp = \alpha_\frakp + \beta_\frakp = \alpha_\frakp + p^3/\alpha_\frakp
\end{equation}
up to sign.  

We then go hunting in \textsc{Magma} by slowly increasing the level and looking for newforms whose Hecke eigenvalues match the value $a_\frakp$ in \eqref{eqn:valueofap} up to sign.  With a candidate in hand, we then compute all good $L$-factors using \eqref{eqn:Asai-lf} to identify a precise match.  The bottleneck in this approach is the computation of systems of Hecke eigenvalues for Hilbert modular forms.

\section{Conclusion}
\label{sec7}

The $1/\pi$ story brings many more puzzles into investigation, as formulas discussed in this note do not exhaust the full set of mysteries. Some of them are associated with the special $_4F_3$ evaluations of $1/\pi$, like the intermediate one in the trio
\begin{align}
\sum_{n=0}^\infty\frac{(\frac12)_n(\frac14)_n(\frac34)_n}{n!^3}
(40n+3)\frac1{7^{4n}}&=\frac{49\sqrt3}{9\pi},
\label{eq:7^4-i}
\\
\sum_{n=0}^\infty\frac{(\frac18)_n(\frac38)_n(\frac58)_n(\frac78)_n}{n!^3(\frac32)_n}
(1920n^2+1072n+55)\frac1{7^{4n}}&=\frac{196\sqrt7}{3\pi},
\label{eq:7^4-ii}
\\
\sum_{n=0}^\infty\frac{(\frac12)_n(\frac18)_n(\frac38)_n(\frac58)_n(\frac78)_n}{n!^5}
(1920n^2+304n+15)\frac1{7^{4n}}&\overset?=\frac{56\sqrt7}{\pi^2}.
\label{eq:7^4-iii}
\end{align}
Here the first equation is from Ramanujan's list \cite[eq.~(42)]{Ra14}, the second one is recently established by Guillera \cite[eq.~(1.6)]{Gu17}, while the third one corresponds to Entry \#15 in Table~\ref{tab1} and is given in~\cite[eq.~(2-5)]{Gu03}.
There is also one formula for $1/\pi^3$, due to B.~Gourevich (2002),
\begin{align}
\sum_{n=0}^\infty\frac{(\frac12)_n^7}{n!^7}
(168n^3+76n^2+14n+1)\frac1{2^{6n}}&\overset?=\frac{32}{\pi^3},
\label{rama5}
\\ \intertext{which shares similarities with Ramanujan's \cite[eq.~(29)]{Ra14}}
\sum_{n=0}^\infty\frac{(\frac12)_n^3}{n!^3}(42n+5)\frac1{2^{6n}}
&=\frac{16}\pi
\label{rama2c}
\end{align}
(observe that $168=42\times4$). And the pattern extends even further with the support of the experimental findings
\begin{align}
\sum_{n=0}^\infty\frac{(\frac12)_n^7(\frac14)_n(\frac34)_n}{n!^9}
(43680n^4+20632n^3+4340n^2+466n+21)\frac1{2^{12n}}
&\overset?=\frac{2048}{\pi^4},
\label{cullen}
\\ \intertext{due to J.~Cullen (December 2010), and}
\sum_{n=0}^{\infty}\frac{(\frac12)_n^5(\frac13)_n(\frac23)_n(\frac14)_n(\frac34)_n}{n!^9}(4528n^4+3180n^3+972n^2+147n+9)\left(-\frac{27}{256}\right)^n
&\overset?=\frac{768}{\pi^4},
\label{rama2d}
\end{align}
due to Yue Zhao~\cite{Zh17} (September 2017).
On the top of these examples there are `divergent' hypergeometric formulas for $1/\pi^3$ and $1/\pi^4$ coming from
`reversing' Zhao's experimental formulas for $\pi^4$ and $\zeta(5)$ in \cite{Zh17}, and corresponding to the hypergeometric data
$$
{}_7F_6\biggl(\begin{matrix}
\tfrac12, \, \tfrac12, \, \tfrac12, \, \tfrac12, \, \tfrac12, \, \tfrac13, \, \tfrac23 \\
1, \, \dots, \, 1 \end{matrix} \biggm| \frac{3^3}{2^2} \biggr)
\quad\text{and}\quad
{}_9F_8\biggl(\begin{matrix}
\tfrac12, \, \tfrac12, \, \tfrac12, \, \tfrac12, \, \tfrac12, \, \tfrac15, \, \tfrac25, \, \tfrac35, \, \tfrac45 \\
1, \, \dots, \, 1 \end{matrix} \biggm| -\frac{5^5}{2^{10}} \biggr),
$$
respectively.  We hope to address the arithmetic-geometric origins of the underlying motives in the near future.


\begin{thebibliography}{99}

\bibitem{AG12}
\textsc{Almkvist, G.}, \textsc{Guillera, J.},
Ramanujan-like series for $1/\pi^2$ and string theory,
\emph{Experiment. Math.} \textbf{21} (2012), no.~3, 223--234.

\bibitem{Asa77}
\textsc{Asai, T.},
On certain Dirichlet series associated with Hilbert modular forms and Rankin's method,
\emph{Math. Ann.} \textbf{226} (1977), 81--94.

\bibitem{BCM15}
\textsc{Beukers, F.}, \textsc{Cohen, H.}, \textsc{Mellit, A.},
Finite hypergeometric functions, \emph{Pure Appl. Math. Q.} \textbf{11}:4 (2015), 559--589.

\bibitem{Magma}  
\textsc{Bosma, W.}, \textsc{Cannon, J.}, \textsc{Playoust, C.}, 
The Magma algebra system.\ I.\ The user language, \emph{J.\ Symbolic Comput.} \textbf{24}:3--4 (1997), 235--265.

% \bibitem{BPPTVY18}
% \textsc{Brumer, A.}, \textsc{Pacetti, A.}, \textsc{Poor, C.}, \textsc{Tornaria, G.}, \textsc{Voight, J.}, \textsc{Yuen, D.Y.},
% On the paramodularity of typical abelian surfaces,
% \emph{Algebra Number Theory} (to appear);
% \emph{Preprint} \href{http://arxiv.org/abs/1805.10873}{\texttt{arXiv:\,1805.10873 [math.NT]}} (2018), 46~pp.

\bibitem{COGP91}
\textsc{Candelas, P.}, \textsc{de la Ossa, X.}, \textsc{Green, P.S.}, \textsc{Parkes, L.},
A pair of Calabi--Yau manifolds as an exactly soluble superconformal theory,
\emph{Nuclear Phys. B} \textbf{359} (1991), no. 1, 21--74.

\bibitem{CC12}
\textsc{Chan, H.H.}, \textsc{Cooper, S.},
Rational analogues of Ramanujan's series for $1/\pi$,
\emph{Math. Proc. Cambridge Philos. Soc.} \textbf{153} (2012), no. 2, 361--383.

\bibitem{clemens} 
\textsc{Clemens, C.~H.}, 
\emph{A scrapbook of complex curve theory}, 2nd ed., Grad.\ Studies in Math., vol.\ 55, Amer.\ Math.\ Soc., Providence, RI (2003).

\bibitem{Coh15}
\textsc{Cohen, H.},
Computing $L$-functions: a survey,
\emph{J. Th\'eor. Nombres Bordeaux} \textbf{27} (2015), no. 3, 699--726.

\bibitem{DV13}
\textsc{Demb\'el\'e, L.}, \textsc{Voight, J.},
Explicit methods for Hilbert modular forms,
\emph{Elliptic curves, Hilbert modular forms and Galois deformations},
Adv. Courses Math. CRM Barcelona, Birkh\"auser/Springer, Basel (2013), 135--198.

\bibitem{codeonline}
\textsc{Demb\'el\'e, L.}, \textsc{Voight, J.},
Code for Asai recognition (2019), available at
\url{http://www.math.dartmouth.edu/~jvoight/HGM_Asai.m}.

\bibitem{DKSSVW18a}
\textsc{Doran, C.F.}, \textsc{Kelly, T.L.}, \textsc{Salerno, A.}, \textsc{Sperber, S.}, \textsc{Voight, J.}, \textsc{Whitcher, U.},
Zeta functions of alternate mirror Calabi--Yau families,
\emph{Israel J. Math.} \textbf{228} (2018), no. 2, 665--705.
%\emph{Preprint} \href{http://arxiv.org/abs/1612.09249}{\texttt{arXiv:\,1612.09249 [math.NT]}} (2018), 29~pp.

\bibitem{DKSSVW18b}
\textsc{Doran, C.F.}, \textsc{Kelly, T.L.}, \textsc{Salerno, A.}, \textsc{Sperber, S.}, \textsc{Voight, J.}, \textsc{Whitcher, U.},
Hypergeometric decomposition of symmetric K3 quartic pencils,
Res.~Math.~Sci.\ (to appear).

\bibitem{padic} 
\textsc{Dwork, B.},
{$p$}-adic cycles,
\emph{Inst. Hautes \'Etudes Sci. Publ. Math.} \textbf{37} (1969), 27--115.

\bibitem{ES}
\textsc{Elkies, N.D.}, \textsc{Sch\"utt, M.},
\emph{K3 families of high Picard rank}, \url{http://www2.iag.uni-hannover.de/~schuett/K3-fam.pdf} (2008).
 
\bibitem{Freitag}
\textsc{Freitag, E.}, \emph{Hilbert modular forms}, Springer-Verlag, Berlin (1990).

%\bibitem{FM16}
%\textsc{Fuselier, J.G.},  \textsc{McCarthy, D.},
%Hypergeometric type identities in the $p$-adic setting and modular forms,
%\emph{Proc. Amer. Math. Soc.} \textbf{144} (2016), no.~4, 1493--1508.
	
\bibitem{geer} 
\textsc{van der Geer, G.}, \emph{Hilbert modular surfaces}, Springer-Verlag, Berlin (1988).

\bibitem{Gr87}
\textsc{Greene, J.},
Hypergeometric functions over finite fields,
\emph{Trans. Amer. Math. Soc.} \textbf{301} (1987), 77--101.

%\bibitem{GK79}
%\textsc{Gross, B.}, \textsc{Koblitz, N.},
%Gauss sums and the $p$-adic $\Gamma$-function,
%\emph{Ann. Math.} \textbf{109} (1979), no. 3, 569--581.

\bibitem{Gu02}
\textsc{Guillera, J.},
Some binomial series obtained by the WZ-method,
\emph{Adv. Appl. Math.} \textbf{29} (2002), no.~4, 599--603. 
% e-print http://arxiv.org/abs/math/0503345.

\bibitem{Gu03}
\textsc{Guillera, J.},
\href{https://projecteuclid.org/euclid.em/1087568026}
{About a new kind of Ramanujan-type series},
\emph{Experiment. Math.} \textbf{12} (2003), no.~4, 507--510.

\bibitem{Gu06}
\textsc{Guillera, J.},
Generators of some Ramanujan formulas,
\emph{Ramanujan J.} \textbf{11} (2006), no.~1, 41-48.

\bibitem{Gu11}
\textsc{Guillera, J.},
A new Ramanujan-like series for $1/\pi^2$,
\emph{Ramanujan J.} \textbf{26} (2011), no.~3, 369--374.

\bibitem{Gu17}
\textsc{Guillera, J.},
More Ramanujan--Orr formulas for $1/\pi$,
New Zealand J. Math. \textbf{47} (2017), 151--160.

%\bibitem{Gu18}
%\textsc{Guillera, J.},
%WZ pairs and $q$-analogues of Ramanujan's series for $1/\pi$ (with an appendix by W.~Zudilin),
%\emph{J. Diff. Equat. Appl}. \textbf{24} (2018), no. 12, 1871--1879.

\bibitem{Gu19}
\textsc{Guillera, J.},
Bilateral sums related to Ramanujan-like series,
\emph{Preprint} \href{https://arxiv.org/abs/1610.04839v2}{\texttt{arXiv:\,1610.04839v2 [math.NT]}} (2019), 13~pp.

\bibitem{Gu20}
\textsc{Guillera, J.},
Bilateral Ramanujan-like series for $1/\pi^k$ and their congruences,
\emph{Preprint} \href{https://arxiv.org/abs/1908.05123}{\texttt{arXiv:\,1908.05123 [math.NT]}} (2019), 18~pp.

\bibitem{GZ12}
\textsc{Guillera, J.}, \textsc{Zudilin, W.},
``Divergent'' Ramanujan-type supercongruences,
\emph{Proc. Amer. Math. Soc.} \textbf{140} (2012), no.~3, 765--777.

\bibitem{GZ13}
\textsc{Guillera, J.}, \textsc{Zudilin, W.},
Ramanujan-type formulae for $1/\pi$: the art of translation,
in ``The Legacy of Srinivasa Ramanujan'', B.C. Berndt \& D. Prasad (eds.),
\emph{Ramanujan Math. Soc. Lecture Notes Ser.} \textbf{20} (2013), 181--195.
%{\tt arXiv:1302.0548v2 [math.NT]}

%\bibitem{He27}
%\textsc{Hecke, E.},
%Theorie der Eisensteinschen Reihen und ihre Anwebdung auf Funktionnentheorie und Arithmetik,
%\emph{Abh. Math. Sem. Hamburg} \textbf{5} (1927), 199--224.

%\bibitem{Hi4}
%\textsc{Hida, H.},
%\emph{Elementary theory of $L$-functions and Eisenstein series},
%London Math. Soc. Student Texts \textbf{26}, Cambridge University Press, Cambridge (1993).

\bibitem{Igusa} 
\textsc{Igusa, J.}, Class number of a definite quaternion with prime discriminant,
\emph{Proc. Nat. Acad. Sci. USA} \textbf{44} (1958), 312--314.

%\bibitem{Ike01}
%\textsc{Ikeda, T.},
%On the lifting of elliptic cusp forms to Siegel cusp forms of degree $2n$,
%\emph{Ann. of Math.} (2) \textbf{154} (2001), 641--681.

\bibitem{Ka76}
\textsc{Katz, N.M.},
$p$-Adic interpolation of real analytic Eisenstein series,
\emph{Ann. of Math.} \textbf{104} (1976), 459--571.

\bibitem{Ka90}
\textsc{Katz, N.M.},
\emph{Exponential sums and differential equations},
Annals of Math. Studies \textbf{124}, Princeton (1990).

\bibitem{Kr02a}
\textsc{Krishnamurthy, M.}, 
Determination of cusp forms on $GL(2)$ by coefficients restricted to quadratic subfields (with an appendix by Dipendra Prasad and Dinakar Ramakrishnan),
\emph{J. Number Theory} \textbf{132} (2012), no. 6, 1359--1384. 

\bibitem{Kr03}
\textsc{Krishnamurthy, M.},
The Asai transfer to $\GL_4$ via the Langlands-Shahidi method, \emph{Int. Math. Res. Not.} (2003), no.~41, 2221--2254.

%\bibitem{Kob80}
%\textsc{Koblitz, N.},
%\emph{$p$-Adic analysis: a short course on recent work},
%London Math. Soc. Lect. Notes Series \textbf{46}, Cambridge University Press, Cambridge (1980).

%\bibitem{Kub}
%\textsc{Kubota, T.},
%\emph{Elementary theory of Eisenstein series},
%Kodansha Ltd. \& John Wiley and Sons, Halsted Press (1973).

%\bibitem{LangMF}
%\textsc{Lang, S.},
%\emph{Introduction to modular forms}, With appendices by D.~Zagier and W.~Feit,
%Springer-Verlag, Berlin (1995).

% \bibitem{LTYZ17}
% \textsc{Long, L.}, \textsc{Tu, F.-T.}, \textsc{Yui, N.}, and \textsc{Zudilin, W.},
% {\em Supercongruences for rigid hypergeometric Calabi--Yau threefolds},
% \emph{Preprint} \href{http://arxiv.org/abs/1705.01663}{\texttt{arXiv:\,1705.01663 [math.NT]}} (2017), 33~pp.

%\bibitem{Miy89}
%\textsc{Miyake, T.},
%\emph{Modular forms},
%Berlin etc., Springer-Verlag (1989).

%\bibitem{Pa81}
%\textsc{Panchishkin,  A.A.},
%{\em Complex valued measures attached to Euler products},  Trudy Sem. Petrovskogo 7 (1981) 239--244 (in Russian)

%\bibitem{Pa82}
%\textsc{Pan\v ci\v skin, A.A.},
%Le prolongement $p$-adique analytique de fonctions $L$ de Rankin I, II.  C. R. Acad. Sci. Paris  294 (1982) 51-53, 227-230.

%\bibitem{Pa91}
%\textsc{Panchishkin,  A.A.},
%{\em Non-Archimedean $L$-functions of Siegel and Hilbert modular forms}, Lecture Notes in Math., {\bf 1471}, Springer-Verlag, 1991, 166p.

\bibitem{Pr92}
\textsc{Prasad, D.},
Invariant forms for representations of $\GL_2$ over a local field,
\emph{Amer. J. Math}. \textbf{114} (1992), no. 6, 1317--1363.

\bibitem{Ra02}
\textsc{Ramakrishnan, D.},
Modularity of solvable Artin representations of $\GO(4)$-type,
\emph{Int. Math. Res. Not.} (2002), no.~1, 1--54.

\bibitem{Ra14}
\textsc{Ramanujan, S.},
Modular equations and approximations to~$\pi$,
\emph{Quart. J. Math. Oxford Ser.}~(2) \textbf{45} (1914), 350--372;
Reprinted in \emph{Collected papers of Srinivasa Ramanujan},
G.\,H.~Hardy, P.\,V.~Sechu Aiyar, and B.\,M.~Wilson (eds.),
Cambridge University Press, Cambridge (1927) \&
Chelsea Publ., New York (1962), 23--39.

\bibitem{RV17}
\textsc{Roberts, D.P.}, \textsc{Rodriguez-Villegas, F.},
%\href{https://doi.org/10.1007/978-3-030-04161-8_33}{Hypergeometric supercongruences},
Hypergeometric supercongruences,
in \emph{2017 MATRIX Annals}, MATRIX Book Ser. \textbf{2}, Springer (2019), 435--439.
%\emph{Preprint} \href{http://arxiv.org/abs/1803.10834}{\texttt{arXiv:\,1803.10834 [math.NT]}} (2018), 5~pp.

\bibitem{RVW}
\textsc{Roberts, D.}, \textsc{Rodriguez-Villegas, F.}, \textsc{Watkins, M.}, 
\emph{Hypergeometric motives}, in preparation.

\bibitem{Tay89}
\textsc{Taylor, R.},
On Galois representations associated to Hilbert modular forms,
\emph{Invent. Math.} \textbf{98} (1989), no. 2, 265--280.

%\bibitem{Vi17}
%\textsc{Rodriguez-Villegas, F.},
%\emph{Hypergeometric motives}, Lecture notes (2017).

%\bibitem{We49}
%\textsc{Weil, A.},
%Numbers of solutions of equations in finite fields,
%\emph{Bull. Amer. Math. Soc.} \textbf{55} (1949), 497.

%\bibitem{We52}
%\textsc{Weil, A.},
%Jacobi sums as ``Gr\"ossencharaktere'',
%\emph{Trans. Amer. Math. Soc.} \textbf{73} (1952), no. 3, 487--495.

\bibitem{Zh17}
\textsc{Zhao, Y.},
A mysterious connection between Ramanujan-type formulas for $1/\pi^k$ and hypergeometric motives,
\url{http://mathoverflow.net/questions/281009/} (September 2017).

%\bibitem{Zu08}
%\textsc{Zudilin, W.},
%Ramanujan-type formulae for $1/\pi$: A second wind?,
%\emph{Modular forms and string duality} (Banff, June 2006),
%N.~Yui, H.~Verrill, and C.\,F.~Doran (eds.),
%Fields Inst. Commun. Ser. \textbf{54},
%Amer. Math. Soc. \& Fields Inst., Providence, RI (2008), 179--188.

%\bibitem{Zu09}
%\textsc{Zudilin, W.},
%Ramanujan-type supercongruences,
%\emph{J. Number Theory} \textbf{129} (2009), no. 8, 1848--1857.

\bibitem{Zu18}
\textsc{Zudilin, W.},
A hypergeometric version of the modularity of rigid Calabi--Yau manifolds,
\emph{SIGMA} \textbf{14} (2018), 086, 16~pages.

\end{thebibliography}
\end{document}